\documentclass[12pt]{article}
\usepackage{amsmath}
\usepackage{amssymb}
\usepackage{amsthm}

\def\dom{\mathrm{dom}\,}
\def\supp{\mathrm{supp}\,}
\def\sigmam{\sigma\hspace{-1mm}_{_M}}
\def\sigman{\sigma\hspace{-1mm}_{_N}}

\newtheorem{theorem}{Theorem}[section]
\newtheorem{lemma}[theorem]{Lemma}
\newtheorem{proposition}[theorem]{Proposition}

\theoremstyle{definition}
\newtheorem{definition}[theorem]{Definition}

\newtheorem{fact}[theorem]{Fact}

\theoremstyle{remark}

\numberwithin{equation}{section}

\begin{document}

\title{Orlicz functions that do not satisfy\\ the $\Delta_2$-condition and high order\\ Gateaux smoothness in $ h_M(\Gamma) $}


%    author one information
\author{Milen Ivanov\footnote{Radiant Life Technologies Ltd., Nicosia, Cyprus, milen@radiant-life-technologies.com, The study  is financed by the European Union-NextGenerationEU, through the National Recovery and Resilience Plan of the Republic of Bulgaria, project SUMMIT BG-RRP-2.004-0008-C01.},
Stanimir Troyanski\footnote{Institute of Mathematics and Informatics, Bulgarian Academy of Sciences, bl. 8, Acad. G. Bonchev Str., 1113 Sofia, Bulgaria, and
Departamento de Matem\'aticas, Universidad de Murcia, Campus de Espinardo, 30100 Murcia, Spain, stroya@um.es, Supported by Grant PID2021-122126NB-C32 funded by
MCIN/AEI/10.13039/501100011033 and by "ERDF A way of making Europe", by the European Union.},
and Nadia Zlateva\footnote{Sofia University, Faculty of Mathematics and Informatics,  5 James Bourchier Blvd., 1164 Sofia,  Bulgaria, zlateva@fmi.uni-sofia.bg, The study is financed by the European Union-NextGenerationEU, through the National Recovery and Resilience Plan of the Republic of Bulgaria, project SUMMIT BG-RRP-2.004-0008-C01.}}

\date{\bigskip \bigskip Dedicated to Rumen Maleev, in memoriam}

\maketitle

\newpage

%    Abstract is required.
\begin{abstract}
 We study Orlicz functions that do not satisfy the $\Delta_2$-condition at zero.
 We prove that for every Orlicz function $M$ such that $\limsup_{t\to0}M(t)/t^p >0$ for some $p\ge1$, there exists a positive sequence $T=(t_k)_{k=1}^\infty$ tending to zero and such that
 $$
   \sup_{k\in\mathbb{N}}\frac{M(ct_k)}{M(t_k)} <\infty,\text{ for all }c>1,
  $$
  that is, $M$ satisfies the $\Delta_2$ condition with respect to $T$.

 Consequently, we show that for each Orlicz function with lower Boyd index $\alpha_M < \infty$ there exists an Orlicz function $N$ such that:

 (a) there exists a positive sequence $T=(t_k)_{k=1}^\infty$ tending to zero such that $N$ satisfies the $\Delta_2$ condition with respect to $T$, and

 (b) the space $h_N$ is isomorphic to a  subspace of $h_M$ generated by one vector.

  We apply this result to find the maximal possible order of G\^ateaux differentiability of  a continuous bump function on the Orlicz space $h_M(\Gamma)$ for $\Gamma$ uncountable.
\end{abstract}

\bigskip\noindent
\textbf{2010 Math Subject Classification:  46B45, 46E30, 49J50}

\section{Introduction}
Since the seminal work of Bonic and Frampton \cite{BF} the exact smoothness of particular classes of Banach spaces has been studied, see e.g. \cite[Chapter 5]{HJ}, \cite[Chapter III.27]{GMZ}.
The exact order of G\^ateaux smoothness of Orlicz spaces $\ell_M(\Gamma)$ over an uncountable index set $\Gamma$ when $M$ satisfies the $\Delta_2$-condition at zero, was completely characterised by Maleev, Nedev and Zlatanov in \cite{MNZ} in terms of the lower Boyd index, e.g. \cite{LT1},
\begin{equation}
    \label{eq:alfa-def}
    \alpha_M := \sup\left\{p:\ \sup\left\{\frac{M(uv)}{u^pM(v)}:\ u,v\in(0,1]\right\} < \infty \right\}.
\end{equation}
From the convexity of the Orlicz function $M$ it follows that $1\le\alpha_M\le\infty$.

Here we show that the characterisation from \cite{MNZ} is still valid in $h_M(\Gamma)$ even without the $\Delta_2$-condition.
\begin{theorem}
    \label{thm:sharp}
    Let $M$ be an Orlicz function, let $\Gamma$ be an    uncountable set and let $p\in\mathbb{N}$. Assume that there exists continuous $p$-times G\^ateaux differentiable bump function on $h_M(\Gamma)$. Then at least one of the following conditions holds:
    \begin{itemize}
        \item[$\mathrm{(a)}$] $p < \alpha_M$;
        \item[$\mathrm{(b)}$] $p = \alpha_M$  and $\displaystyle \lim_{t\to0}M(t)/t^p=0$;
        \item[$\mathrm{(c)}$] $p$ is even and $M$ is equivalent to $t^p$ at zero. (Obviously in this case also $p = \alpha_M$.)
    \end{itemize}
\end{theorem}
Theorem~\ref{thm:sharp} is sharp. In case (a) there even exists  equivalent $p$-times Ft\'echet differentiable norm on $h_M(\Gamma)$, see \cite[Corollary~12]{MT}. Case (b) is a modification of \cite[Lemma 2 and Lemma 4]{MT}, see also \cite[Remark~1]{MNZ}.
 Of course, in the  (c) case $h_M(\Gamma)$ is just $\ell_p(\Gamma)$ which for an even $p$ is infinitely smooth.

The proof of the exact order of G\^ateaux smoothness of Orlicz spaces $\ell_M(\Gamma)$ over an uncountable index set $\Gamma$ when $M$ satisfies the $\Delta_2$-condition at zero due to \cite[Corollary~1]{MNZ}, relies on Stegall Variational Principle \cite{stegal1,stegal2}. The Stegall Variational Principle applies when the domain of the minimized function is a dentable set. When $M$ satisfies the $\Delta_2$-condition at zero the authors of \cite{MNZ} use that the unit vectors form a boundedly complete basis of $\ell_M$ and, therefore, $\ell_M$ has the Radon-Nikod\'ym property. If, however, $M$ fails the $\Delta_2$-condition at zero then $h_M$ contains a subspace isomorphic to $c_0$ and, therefore, $h_M$ fails also the Radon-Nikod\'ym property and Stegall Variational Principle cannot be applied directly.

We will briefly explain the new technique we use to circumvent the lack of $\Delta_2$-condition in Theorem~\ref{thm:sharp}.

Elaborating on an idea of Leung \cite{leung} we show that there is a positive sequence $(t_k)_{k=1}^\infty$ tending to zero such that $M$ satisfies a kind of $\Delta_2$-condition relative to this sequence. More precisely:
\begin{definition}
  \label{dfn:relative-del-2}
  Let $M$ be a non-degenerate Orlicz function and let $T=(t_k)_{k=1}^\infty$ be a  positive sequence tending to $0$. We say that $M$ satisfies the $\Delta_2$-condition \emph{with respect to} $T$ if
  \begin{equation}
    \label{eq:relative-del-2}
    \sup_{t\in T}\frac{M(ct)}{M(t)} < \infty
  \end{equation}
  for all $c>1$.
\end{definition}
The merit of the concept defined above is that if $M$ satisfies the $\Delta_2$-condition with respect to a sequence $T$ then the unit vector basis $(e_\gamma)_{\gamma\in\Gamma}$ in $h_M(\Gamma)$ is boundedly complete with respect to this $T$. More precisely, let $T=(t_k)_{k=1}^\infty$ be a  positive sequence  tending to~$0$ and let
$\left(\{\pm T\}\cup\{0\}\right)^\Gamma$ be the set of all functions $x=x(\gamma)$ defined on $\Gamma$ with values in $\{\pm T\}\cup\{0\}$. In other words,
\begin{equation}
  \label{eq:T-Gamma-def}
  x \in  \left(\{\pm T\}\cup\{0\}\right)^\Gamma \iff |x(\gamma)|\in T\cup\{0\},\ \forall \gamma\in\Gamma.
\end{equation}
\begin{fact}
  \label{fact:bound-complete}
  Let $T=(t_k)_{k=1}^\infty$ be a  positive sequence tending to $0$ and let $M$ satisfy the $\Delta_2$-condition with respect to $T$. Let $(e_\gamma)_{\gamma\in\Gamma}$ be the unit vector basis of $h_M(\Gamma)$. Let $x \in  \left(\{\pm T\}\cup\{0\}\right)^\Gamma$.
  Assume that for any finite subset $A$ of $\Gamma$
  \begin{equation}
    \label{eq:bound-complete}
    \left\| \sum_{\gamma\in A} x(\gamma) e_\gamma  \right\| \le 1.
  \end{equation}
  Then $x\in h_M(\Gamma)$.
\end{fact}

Indeed,  \eqref{eq:bound-complete} is equivalent to $\sigmam(x)\le 1$.
  Let $s_i := \sup_{t\in T} M(it)/M(t)$, where $i\in\mathbb{N}$. From \eqref{eq:relative-del-2} it follows that $s_i < \infty$. So, $M(it_k) \le s_iM(t_k)$ and, therefore,
  $
    \sigmam (ix) \le s_i \sigmam (x) \le s_i < \infty,
  $
  for all $i\in\mathbb{N}$. This means that $x\in h_M(\Gamma)$.

Our new insight is recorded in the following theorem.
\begin{theorem}
    \label{thm:t-seq}
    Let the Orlicz function $M$ be such that
    \begin{equation}
      \label{eq:t-seq-cond}
      \limsup_{t\to0}\frac{M(t)}{t^p} > 0
    \end{equation}
    for some $p\ge 1$.
    Then there are constants $a>0$ and a positive sequence $T=(t_k)_{k=1}^\infty$ tending to $0$ such that $M$ satisfies the $\Delta_2$-condition with respect to $T$ and
    \begin{equation}
        \label{eq:t-seq-lim}
        \liminf_{k\to\infty}\frac{M(t_k)}{t_k^p} \ge a\limsup_{t\to0}\frac{M(t)}{t^p}.
    \end{equation}
\end{theorem}
In Section~\ref{sec:reduce} we show that if $p>\alpha_M$ then there is a subspace of $h_M(\Gamma)$ generated by one vector and isometric to $h_N(\Gamma)$, where $N$ is an Orlicz function such that $\limsup_{t\to0}N(t)/t^p = \infty$. Thus, \eqref{eq:t-seq-cond} holds for $N$. So, $N$ satisfies the $\Delta_2$-condition with respect to a suitable $T$. Leung \cite{leung} has provided an example of an Orlicz function $M$ with $\alpha_M = \infty$ for which there is a sequence $T$ such that  $M$ satisfies the $\Delta_2$-condition with respect to $T$. Of course, if $M(2t)/M(t) \to \infty$ as $t\to0$, such $T$ does not exist.

The problem of finding  spaces $h_N$  isomorphic to subspaces of a given space $h_M$ generated by
one vector,  and the problem of finding spaces $h_N$
just isomorphic to subspaces of  $h_M$, are quite different.
For a given Orlicz function $M$, those Orlicz functions $N$ for which $h_N$ is a
subspace of $h_M$ are characterized in \cite[Theorem~4.a.8]{LT1}. This condition is
effective. In very precise terms (non empty closed interval) are characterised
those $\ell_p$ which are isomorphic to a subspace of $h_M$ \cite[Theorem~4.a.9]{LT1}. The
characterisation \cite[Proposition~5]{HT}  of those Orlicz functions $N$ for which
$h_N$ is isomorphic to a subspace of $h_M$ generated by one vector, is harder to
apply, see \cite[Proposition~7]{HT} and \cite[Theorem~B]{HRS}. However, it is
relatively easy to give an example of an Orlicz function $M$ such that $h_M$ has no subspaces generated by one vector isomorphic to $\ell_p$ for any
$p$. For instance, if
$$
  M(t) = \left(\frac{t}{\log t}\right)^2,\quad\mbox{ for }t\in(0,t_0),
$$
for some suitable $t_0>0$ and extended in a convex fashion for $t>t_0$, then $h_M$ has no subspaces generated by one vector isomorphic to $\ell_p$ for any $p\ge1$, see \cite[Example~2]{RS}.

The way the sequence $T$ provided by Theorem~\ref{thm:t-seq} can be used is as follows. If $\alpha_M>1$ then the unit vector basis $(e_\gamma)_{\gamma\in\Gamma}$ of $h_M(\Gamma)$ is shrinking, i.e. the conjugate system $(e^*_\gamma)_{\gamma\in\Gamma}$ generates the dual $h_M(\Gamma)^*$. Thus the intersection of the unit ball of $h_M(\Gamma)$ with the set $\left(\{\pm T\}\cup\{0\}\right)^\Gamma $, see \eqref{eq:T-Gamma-def}, is weekly compact, cf. \cite[Section 3, Examples]{FST}. Therefore, the latter intersection is dentable and Stegall Variational Principle can be used on it as in \cite{HT,MNZ}. However, we prefer to use direct induction of the type originating from~\cite{BF} and elaborated in \cite{MT}. Note that another way for completing this step is the Ekeland Variational Principle with respect to an appropriate metric as in \cite{Stankov} and restricted to $\left(\{\pm T\}\cup\{0\}\right)^\Gamma $, because without $\Delta_2$ that metric is not complete on the whole of $h_M(\Gamma)$. The special perturbation method from \cite{Topalova-Zlateva} should also work if restricted in the same way.

It is worth mentioning here that the exact order of G\^ateaux smoothness of $\ell_p(\mathbb{N})$ for $p>2$ is not known and it is an interesting and apparently tough open problem. Therefore, it is not possible at present to derive Theorem~\ref{thm:sharp} from the countable case.

The paper is organised as follows. After Section~\ref{sec:prelim} of preliminaries, we prove in Section~\ref{sec:nc} a necessary condition for smoothness which, when applied to a suitable subspace, provided by Section~\ref{sec:reduce}, yields Theorem~\ref{thm:sharp}.

\section{Preliminaries}\label{sec:prelim}

    Here we recall some definitions and known facts that we use.

    On a Banach space $X$ the function $f:X\to\mathbb{R}$ is $p$-times G\^ateaux differentiable, $p\in\mathbb{N}$, if
    $$
      f(x+ ty) = f(x) + P(ty) + o_y(t^p),
    $$
    where $P$ is a  bounded polinomial of degree at most $p$, such that $P(0)=0$,  and $o_y(t^p)$ means any function $\omega :X\times\mathbb{R}\to\mathbb{R}$ such that
    $$
        \lim _{t\to0}{\omega(y,t)\over t^p} = 0,\quad\forall y\in X,
    $$
  for details see~\cite[Section~1.4]{HJ}.

    The function $b:X\to\mathbb{R}$ is called {\em bump} if its support set
    $$
        \supp b := \{x:\ b(x)\neq0\}
    $$
    is bounded. Here we assume non-trivial bumps, that is, such with non-empty support.

    A non-degenerate Orlicz function $M$, that is,  an even convex function on $\mathbb{R}$ such that $M(0) = 0$ and $M(t)>0$ for $t\neq 0$, is always assumed throughout the paper. Recall, e.g. \cite{LT1}, that $M$ satisfies the $\Delta_2$-condition at zero if
    $$
    \limsup_{t\to0}\frac{M(2t)}{M(t)} < \infty,
    $$
    which is equivalent to
    $$
      \sup_{t\in(0,1]}\frac{M(ct)}{M(t)} < \infty,\text{ for all } c>1.
    $$

    The Orlicz space $\ell_M(\Gamma)$ is the set of all bounded real-valued functions on a set $\Gamma$ such that $\sigmam(x/\lambda) < \infty$ for some $\lambda>0$, where
    $$
        \sigmam(x) := \sum_{\gamma\in\Gamma} M(x(\gamma)).
    $$
    The summation is in the sense of the supremum over the sums on all finite subsets of~$\Gamma$.

    We denote for $x\in \ell_M(\Gamma)$
    $$
        \supp x := \{ \gamma\in\Gamma:\ x(\gamma)\neq 0\}.
    $$
    It is clear that $\supp x$ is countable for each $x\in\ell_M(\Gamma)$.

    The Luxembourg norm of the Orlicz space $\ell_M(\Gamma)$ is given by the Minkowski functional of the set $\{x\in\ell_M(\Gamma):\ \sigmam(x) \le 1\}$, that is,
    $$
        \|x\| := \inf \{\lambda > 0:\ \sigmam(x/\lambda) \le 1\}.
    $$
    Then $\ell_M(\Gamma)$ is a Banach space. We are interested in its closed subspace
    $$
        h_M(\Gamma) := \{x\in \ell_M(\Gamma):\ \sigmam(x/\lambda) < \infty,\ \forall \lambda > 0\}.
    $$
    It is easy to see that $h_M(\Gamma)$ is the subspace of $\ell_M(\Gamma)$ generated by the unit vector basis $(e_\gamma)_{\gamma\in\Gamma}$. Clearly, $(e_\gamma)_{\gamma\in\Gamma}$ is a symmetric basis in $h_M(\Gamma)$. With $(e_\gamma^*)_{\gamma\in\Gamma}$ we denote the conjugate system to the basis $(e_\gamma)_{\gamma\in\Gamma}$, that is, $e_\gamma^*(x) = x(\gamma)$.

It follows directly from the definition of the canonical norm in $\ell _M$ that for a sequence $(x_n)_{n=1}^\infty \subset \ell _M$
\begin{equation}
    \label{eq:i-sigmas}
  \lim_{n\to\infty} \|x_n\| = 0 \iff
  \lim_{n\to\infty} \sigmam(ix_n) = 0,\quad\forall i\in\mathbb{N}.
\end{equation}
Indeed, $\sigmam(ix) \le 1 \iff \|ix\| \le 1 \iff \|x\| \le 1/i$. On the other hand, $\|x\| \le \varepsilon/i\iff \sigmam (ix/\varepsilon)\le 1$. From convexity $\sigmam(\varepsilon y) \le \varepsilon\sigmam(y)$ for $\varepsilon\in[0,1]$, so $\sigmam(ix) \le \varepsilon \sigmam (ix/\varepsilon) \le \varepsilon$.

Note that if $M$ satisfies the $\Delta_2$-condition at zero, then $x_n\to0\iff \sigmam(x_n)\to0$. This is the main technical obstacle we here overcome.

An important part of the proof of Theorem~\ref{thm:sharp} is a reduction to a subspace. In our case due to lack of $\Delta_2$-condition this reduction is a bit more technical than in \cite{MNZ}. We do it Section~\ref{sec:reduce}, but here we will introduce the needed concepts. The following definition is well-known, see e.g. \cite{LT1}.
\begin{definition}
  Let $X$ be a Banach space with symmetric basis $(u_\gamma)_{\gamma\in\Gamma}$. Let $A$ be an index set and let for each $\alpha\in A$ be given a sequence $(\alpha_n)_{n=1}^\infty$ of elements of~$\Gamma$. Let also these sequences be disjoint, i.e. $\{\alpha_n:\ n\in\mathbb{N}\}\cap\{\beta_n:\ n\in\mathbb{N}\} = \emptyset$ if $\alpha\neq\beta$.

  Then for any $z\in X\setminus\{0\}$ represented as
  $$
   z = \sum_{n=1}^\infty a_n u_{\gamma_n},
  $$
  where $\gamma_n\neq\gamma_m$ for $n\neq m$, the basic set $(v_\alpha^{(z)})_{\alpha\in A}$ defined by
  $$
    v_\alpha^{(z)} = \sum_{n=1}^\infty a_nu_{\alpha_n}
  $$
  is called a \emph{block basis generated by the vector} $z$ and the subspace spanned by this block basis is called \emph{subspace generated by} $z$.
\end{definition}
The work \cite{HT} characterised the Orlicz subspaces of $h_M(\Gamma)$.
\begin{proposition}\cite[Proposition 5]{HT}
    Let $M$ and $N$ be Orlicz functions and let $\Gamma$ be an uncountable set. Then $h_M(\Gamma)$ contains a subspace $Z$ isomorphic to $h_N(\Gamma)$ if and only if there exists a vector $\displaystyle z = \sum_{n=1}^\infty a_ne_{\gamma_n} \in h_M(\Gamma)$ such that the corresponding block basis is a basis of $Z$ and $N=N(t)$ is equivalent at zero to
    $$
        Q(t) := \sum_{n=1}^\infty M(a_nt).
    $$
\end{proposition}
Here for the sake of completeness of the presentation we will prove one basic and easy part of the above characterisation.
    \begin{lemma}
  \label{lem:subspace}
  Let $z\in h_M(\Gamma)$, $z\neq0$, and let
  $$
    N(t) := \sigmam(tz).
  $$
  Then $h_N(\Gamma)$ is isometric to a  subspace of $h_M(\Gamma)$ generated by $z$.
\end{lemma}
\begin{proof}
    Let
    $$
        \mathrm{supp}\, z \subset \{\gamma_n:\ n\in\mathbb{N}\}.
    $$

We identify $h_M(\Gamma)$ with $h_M(\Gamma\times \mathbb{N})$ and consider
$$
  U: h_N(\Gamma) \to h_M(\Gamma\times \mathbb{N})
$$
defined by
$$
  (Ux)(\gamma,n) = x(\gamma) z(\gamma_n).
$$
Then
\begin{eqnarray*}
  \sigmam (Ux) &=& \sum_{\gamma \in \Gamma}\sum _{n\in \mathbb{N}} M(x(\gamma) z(\gamma_n))\\
  &=& \sum_{\gamma \in \Gamma} \sigmam (x(\gamma) z) \\
  &=& \sum_{\gamma \in \Gamma} N (x(\gamma)) \\
  &=& \sigman (x).
\end{eqnarray*}
\end{proof}

We also recall the variant of Bonic-Frampton Lemma \cite{BF} (see also \cite[Corollary~3.59]{HJ}) we need.

\begin{lemma}
    \label{lem:poly-alg}
    Let $n\in\mathbb{N}$ and let
    \begin{equation}
        \label{eq:Mn-zero}
        \liminf_{t\to0}\frac{M(t)}{t^{n-1}} = 0.
    \end{equation}
    Then for any bounded polynomial $P$ on $h_M$ of degree
    $$
        \mathrm{deg}\, P < n,
    $$
    such that $P(0) = 0$ there is at most countable $A\subset \Gamma$ such that if $\gamma\not\in A $ then $P(te_\gamma)=0$ for all $t\in\mathbb{R}$.
\end{lemma}
\begin{proof}
    Since a polynomial is a finite sum of forms, it is actually enough to prove that $P(e_\gamma) = 0$ for $\gamma\not\in A $, because then we will simply unite the respective $A$'s for each of the forms comprising $P$ and use that the latter are homogeneous.

    We do induction on $k := \mathrm{deg}\, P$. If $k=0$ the statement is trivial. Let it be true for all polynomials of degree less than $k-1$. Let $P$ be a $k$-form, $k<n$. Then
    $$
        P(x+y) = P(x) + P_1(x,y) + P(y),\quad\forall x,y\in h_M,
    $$
    where $P_1(x,0) = 0$ and the degree of $P_1$ with respect to $y$ is less than $k-1$. So, by induction hypothesis, for each $x\in h_M$ there is countable $A_x\subset \Gamma$ such that  $P_1(x,e_\gamma)=0$ for each $\gamma\not\in A_x$, so $P_1(x,te_\gamma)=0$ for all $t$. Since $P$ is $k$-homogeneous $P(te_\gamma) = t^kP(e_\gamma)$ and we get
    \begin{equation}
        \label{eq:ind}
        P(x+te_\gamma) = P(x) + t^kP(e_\gamma),\quad\forall \gamma\in\Gamma\setminus A_x.
    \end{equation}
    Assume now that $P(e_\gamma)\neq0$ for uncountably many $\gamma\in\Gamma$. Then there is $r>0$ such that $|P(e_\gamma)| \ge r$ for uncountably many $\gamma$'s. By multiplying $P$ by a non-zero constant we may assume that
    \begin{equation}
        \label{eq:P-ge-1}
        P(e_\gamma) \ge 1,\quad\forall\gamma\in B,
    \end{equation}
    where $B\subset\Gamma$ is uncountable.

    We will show that $\sup_{\|x\|\le1}P(x)=\infty$ and get a contradiction.

    Let $\varepsilon \in (0,1)$ be arbitrary. From \eqref{eq:Mn-zero} there is $t\in(0,\varepsilon)$ such that
    $$
        M(t) < \varepsilon t^{n-1} \le \varepsilon t^k.
    $$
    Let
    $$
        m := \left[\frac{1}{\varepsilon t^k}\right].
    $$
    We iterate on \eqref{eq:ind} and \eqref{eq:P-ge-1} to get $x_0,\ldots,x_m$ as follows. Set $x_0=0$. If $x_j$ is constructed, take $\gamma_j\in B\setminus (A_{x_j}\cup \mathrm{supp}\, x_j)$ and let $x_{j+1} = x_j+te_{\gamma_j}$. So,
    $$
        P(x_{j+1}) \ge P(x_j) + t^k,
    $$
    and, therefore,
    $$
        P(x_m) \ge mt^k \ge \left(\frac{1}{\varepsilon t^k} - 1\right)t^k = \frac{1}{\varepsilon} -t^k > \frac{1}{\varepsilon} - \varepsilon^k \to \infty,
    $$
    as $\varepsilon\to0$.
    On the other hand, since
    $$
        x_m = \sum_{j=0}^{m-1} t e_{\gamma_j},
    $$
    and $\gamma_j$'s do not repeat, we have
    $$
        \sigmam(x_m) = mM(t) < m\varepsilon t^k \le \left(\frac{1}{\varepsilon t^k} \right)\varepsilon t^k = 1,
    $$
    so $\|x_m\| \le 1$.
\end{proof}

\section{Necessary Condition for Smoothness}\label{sec:nc}
In this section we prove the following necessary condition.
\begin{proposition}
  \label{prop:main}
  Let $p\in\mathbb{N}$.
  If there is continuous $p$-times G\^ateaux differentiable bump function on $h_M(\Gamma)$, then either
  \begin{equation}
      \label{eq:ddd}
      \lim_{t\to 0}\frac{M(t)}{t^p} = 0,
  \end{equation}
  or $p$ is even and $M$ is equivalent to $t^p$ at zero, that is, $h_M(\Gamma)$ and $ \ell_p(\Gamma)$ are isomorphic.
\end{proposition}
We start with:
\begin{proof}[\textbf{\emph{Proof of Theorem~\ref{thm:t-seq}}}]
  Let
  $$
    \varphi(t) := \frac{M(t)}{t^p},
  $$
  so that $\varphi$ is continuous on $(0,1]$ and
  $$
    \limsup_{t\to0} \varphi(t) > 0,
  $$
  because of \eqref{eq:t-seq-cond}. We can choose a sequence $t_k\to0$ which satisfies \eqref{eq:t-seq-lim} and
  \begin{equation}
    \label{eq:M-le-tk-prim}
    \frac{\varphi(t)}{\varphi(t_k)} \le s,\quad\forall t\in[t_k,1],\ \forall k\in\mathbb{N}
  \end{equation}
  for some $s > 0$.

  Indeed, if $\varphi(t)$ is bounded on $(0,1]$, say, $\varphi(t) \le b$ for all $t\in (0,1]$, take some sequence $t_k\to0$ such that $\varphi(t_k) \ge r >0$ for all $k\in\mathbb{N}$, meaning that \eqref{eq:t-seq-lim} is fulfilled, but also $\varphi(t)/\varphi(t_k) \le b/r$ for all $t\in(0,1]$ and  $k\in\mathbb{N}$, so \eqref{eq:M-le-tk-prim} holds.

  If, on the other hand, $\varphi(t)$ is unbounded, take $t_k\in[k^{-1},1]$ such that $\varphi(t_k)=\max\{\varphi(t):\ t\in[k^{-1},1]\}$. Obviously, $\varphi(t_k)\to\infty$, so $t_k\to0$, and \eqref{eq:t-seq-lim} is true, but also $\varphi(t)/\varphi(t_k) \le 1$ for all $t\in[t_k,1]$, and  \eqref{eq:M-le-tk-prim} is true.

  Now, from \eqref{eq:M-le-tk-prim} it easily follows that $M$ satisfies the $\Delta_2$-condition with respect to $T=(t_k)_{k=1}^\infty$. Indeed, fix $c>1$ and let $T_c = \{t\in T:\ tc \le 1\}$. Since the set $T\setminus T_c$ is finite, it is enough to prove \eqref{eq:relative-del-2} for $T_c$. But for $t\in T_c$ the estimate \eqref{eq:M-le-tk-prim} gives $\varphi(ct) \le s\varphi(t)$, that is, $M(ct) \le sc^p M(t)$, so
  $$
  \sup_{t\in T_c}\frac{M(ct)}{M(t)} \le sc^p.
  $$
\end{proof}

We will first prove Proposition~\ref{prop:main} under the additional assumption \eqref{eq:ad-as}, because after that the general case easily follows.

\begin{proposition}
  \label{prop:add}
  Let $M$ be such that
  \begin{equation}
    \label{eq:ad-as}
    \lim_{t\to0}\frac{M(t)}{t^{p-1}} = 0.
  \end{equation}
  If there is a continuous $p$-times G\^ateaux differentiable bump function on $h_M(\Gamma)$, then either \eqref{eq:ddd} is fulfilled or $p$ is even and $M$ is equivalent to $t^p$ at zero.
\end{proposition}

\begin{proof}
  Assume the contrary, that is, \eqref{eq:ddd} fails and, if $p$ is even, $M$ is not equivalent to $t^p$ at zero; yet there is continuous $p$-times G\^ateaux differentiable bump function $b$ on $h_M(\Gamma)$.

  By considering instead of $b$ the bump $b(r(x-x_0))$, where $x_0$ is such that $b(x_0)\neq0$ and $r>0$ is large enough, we can assume without loss of generality that $0\in\supp b\subset B_{h_M(\Gamma)}$, the unit ball of $h_M(\Gamma)$. For technical reasons we introduce the function
  $$
    f(x) := \begin{cases}
          b^{-2}(x)-b^{-2}(0),& \mbox{if }x\in\supp b,\\
          +\infty,& \mbox{otherwise}.
      \end{cases}
  $$
Clearly, $f$ is a continuous function from $h_M(\Gamma)$ to $[0,\infty]$ which is $p$-times G\^ateaux differentiable at each point of its domain
$$
  \dom f := \{x\in h_M(\Gamma):\ f(x)<\infty\} = \supp b,
$$
and $f(0) = 0$.

Because \eqref{eq:ddd} fails, \eqref{eq:t-seq-cond} holds and we can apply Theorem~\ref{thm:t-seq} to get a sequence $T=(t_k)_{k=1}^\infty$ tending to $0$ such that $M$ satisfies the $\Delta_2$-condition with respect to $T$ and \eqref{eq:t-seq-lim} holds.

\medskip

\noindent\textsc{Claim.} For each $x\in\dom f$ there is uncountable $B(x)\subset\Gamma\setminus\supp x$ such that for each $\gamma\in B(x)$ there is $s=s(x,\gamma)\in\{-1,1\}$ such that
\begin{equation}
  \label{eq:f-sigma-le}
  f(x + st_ke_\gamma) - f(x) < \sigmam(x + st_ke_\gamma) - \sigmam(x)
\end{equation}
for all $k$ large enough.

Indeed, because $f$ is $p$-times G\^ateaux differentiable at $x\in\dom f$,
$$
  f(x+ ty) = f(x) + P(ty) + t^p Q(y) + o_y(t^p),
$$
where $P$ is a  polinomial of degree at most $(p-1)$ such that $P(0)=0$ and $Q$ is a $p$-form.

From \eqref{eq:ad-as} and   Lemma~\ref{lem:poly-alg}  there is a countable $A\subset\Gamma$ such that $P(te_\gamma) = 0$ for all $\gamma\in \Gamma\setminus A$ and $t\in\mathbb{R}$. We can assume $\supp x\subset A$, because the former is countable. So, for any $\gamma\in \Gamma\setminus A$
\begin{equation}
  \label{eq:o-malko}
  f(x+ te_\gamma) - f(x) = t^p Q(e_\gamma) + o_\gamma(t^p),\quad \sigmam(x + te_\gamma) - \sigmam(x) = M(t).
\end{equation}
We can immediately conclude if the set $\{\gamma:\ Q(e_\gamma) \neq0\}$ is countable. Indeed, we may assume that $A$ contains that set as well. Then for any $\gamma\in B(x) := \Gamma\setminus A$
$$
  \lim_{t\to 0} \frac{f(x+te_\gamma) - f(x)}{t^p} = 0,
$$
while $\liminf_{k\to\infty} M(t_k)/t_k^p > 0$ by \eqref{eq:t-seq-lim} and the assumption that \eqref{eq:ddd} fails. So, \eqref{eq:f-sigma-le} is fulfilled for $s=1$ and all $k$ large enough.

Let now $\{\gamma:\ Q(e_\gamma) \neq0\}$ be uncountable set. Then there is  uncountable $B(x)\subset \Gamma\setminus A$ such that all $Q(e_\gamma)$'s have the same sign on $B(x)$. Set $s= -\mathrm{sign}\,Q(e_\gamma)$. So, \eqref{eq:o-malko} can be rewritten as:
$$
  f(x+ st_ke_\gamma) - f(x) = s^pt_k^p Q(e_\gamma) + o_\gamma(t_k^p),\quad \sigmam(x + st_ke_\gamma) - \sigmam(x) = M(t_k).
$$

We have three distinct cases to consider:

(i) $p$ odd,

(ii) $p$ even and $\limsup_{t\to0}M(t)/t^p = \infty$, and

(iii) $p$ even and $\limsup_{t\to0}M(t)/t^p < \infty$.

In the case (i) we have $s^p=s$, so $r_\gamma := s^p Q(e_\gamma) < 0$ for $\gamma\in B(x)$ and, therefore,  $f(x+ st_ke_\gamma) - f(x) < 0$ for any fixed $\gamma\in B(x)$ and all $k\in\mathbb{N}$ large enough, because $\lim_{k\to\infty} (f(x+ st_ke_\gamma) - f(x))/ t_k^p = r_\gamma$. On the other hand $M(t_k) > 0$. Hence,  \eqref{eq:f-sigma-le} is fulfilled.

In the case (ii) because of \eqref{eq:t-seq-lim} we have that $M(t_k)/t_k^p\to\infty$, while  $|f(x+ st_ke_\gamma) - f(x)|/t_k^p$ is bounded for fixed $\gamma\in B(x)$, so \eqref{eq:f-sigma-le}  again is easily fulfilled.

The case (iii) is actually impossible under our present assumptions. Indeed, by assumption $M$ is not equivalent to $t^p$ at zero, so $\liminf_{t\to0} M(t)/t^p = 0$, but then Lemma~\ref{lem:poly-alg} implies $Q(e_\gamma)=0$ for all but countably many $\gamma$'s, contradiction.

\medskip

For convenience we assume that the sequence $T=(t_k)_{k=1}^\infty$ is strictly decreasing.
For $x\in\dom f$ and $\gamma\in B(x)$ we denote by $N(x,\gamma)$ the set of all positive integers $k$ for which \eqref{eq:f-sigma-le} holds for some $s=s(x,\gamma)$. From the \textsc{Claim} we get that $N(x,\gamma)\neq\emptyset$. We denote by $m(x)$ the smallest positive integer in the set $\displaystyle\bigcup_{\gamma\in B(x)} N(x,\gamma)$. We choose a $\mu(x)\in B(x)$ in such a way that
$$
  m(x) \in N(x,\mu(x)).
$$
Set $x_0 = 0$. Since $f(x_0)=0$, of course $x_0\in\dom f$. Using the \textsc{Claim} inductively, we choose sequences $(x_n)_{n=1}^\infty\subset\dom f$, $(\gamma_n)_{n=1}^\infty\subset\Gamma$, $(s_n)_{n=1}^\infty\subset\{-1,1\}$ such that for all $n\in\mathbb{N}$
\begin{equation}
   \label{eq:3?}
   \gamma_n = \mu(x_{n-1}),
\end{equation}
\begin{equation}
   \label{eq:3??}
   \nu_n = t_{m(x_{n-1})},
\end{equation}
$$
        x_{n} := x_{n-1} + s_n\nu_ne_{\gamma_n},
$$
\begin{equation}
        \label{eq:Tj-def}
        f(x_n) - f(x_{n-1}) \le \sigmam(x_n ) - \sigmam(x_{n-1}).
\end{equation}
Since all $x_n$'s are in the domain of $f$ which is  within the unit ball of $h_M(\Gamma)$, we get $\|\sum_{j=1}^n s_j\nu_je_{\gamma_j}\|=\|x_n\|\le 1$ for all $n\in\mathbb{N}$. From Fact~\ref{fact:bound-complete} it follows that the series  $\sum_{j=1}^\infty s_j\nu_je_{\gamma_j}$ converges in norm to some $\bar x\in h_M(\Gamma)$. Iterating on \eqref{eq:Tj-def} we get $f(x_n)-f(x_0) \le \sigmam(x_n) - \sigmam(x_0)$, that is, $f(x_n) \le \sigmam(x_n) $, for all $n\in\mathbb{N}$. From the continuity of $f$ and $\sigmam(x_n)\le 1$ we get $f(\bar x) \le 1$, so $\bar x\in\dom f$.

Pick $\beta\in B(\bar x)\setminus \supp \bar x$. Let $s = s(\bar x,\beta)$ and $l\in N(\bar x, \beta)$. By definition this means $f(\bar x +st_le_\beta) - f(\bar x) < \sigmam(\bar x +st_le_\beta) - \sigmam(\bar x)$. But the latter is equal to $M(t_l)$, so
\begin{equation}
   \label{eq:3Pound}
   f(\bar x +st_le_\beta) - f(\bar x) < M(t_l).
\end{equation}
On the other hand, $f(x_n +st_le_\beta) - f(x_n) \to f(\bar x +st_le_\beta) - f(\bar x)$, as $n\to\infty$. This and \eqref{eq:3Pound} imply $f(x_n +st_le_\beta) - f(x_n) < M(t_l)$ for all $n$ large enough. Note that $\sigmam(x_n +st_le_\beta) - \sigmam(x_n) = M(t_l)$, because $\beta\not\in\supp \bar x \supset \supp x_n$. So, $l\in N(x_n,\beta)$ for all large $n$'s. This implies $l\ge m(x_n)$ and, therefore, $m(x_n) \ge t_l > 0$. From \eqref{eq:3??} we deduce that $\nu_n\ge t_l > 0$ for all $n$ large enough, which contradicts the convergence of the series $\sum_{j=1}^\infty s_j\nu_j e_{\gamma_j}$.
\end{proof}

\begin{proof}[\emph{\textbf{Proof of Proposition~\ref{prop:main}}}]
  We do induction on $p$.

  If $p=1$ then \eqref{eq:ad-as} is just $M(t)\to 0$ as $t\to 0$, which is so, and Proposition~\ref{prop:add} shows that the statement is true in this case.

  Let the statement be true for $p-1$.

  If \eqref{eq:ad-as} is fulfilled then by Proposition~\ref{prop:add} the statement is true also for $p$.

  Assume that \eqref{eq:ad-as} is not fulfilled and there is a continuous $p$-times G\^ateaux differentiable bump function $b$ on $h_M(\Gamma)$. Since $b$ is also $(p-1)$-times G\^ateaux differentiable, from the induction hypothesis $M$ must be equivalent to $t^{p-1}$ at zero, but this implies \eqref{eq:ddd} and the induction is complete.
\end{proof}

\section{Reduction to a subspace}\label{sec:reduce}

In this final section we show how to derive Theorem~\ref{thm:sharp} from Proposition~\ref{prop:main}.
\begin{lemma}
  \label{lem:one-vector}
If
\begin{equation}
 \label{eq:mnz}
 \sup\left\{\frac{M(uv)}{u^pM(v)}:\ u,v\in (0,1]\right\} = \infty,
\end{equation}
then there is a $z\in h_M$ such that
\begin{equation}
 \label{eq:simple-cond}
 \limsup_{t \to 0} \frac{\sigmam(tz)}{t^p} = \infty.
\end{equation}
\end{lemma}

\begin{proof}

If
$$
  \limsup_{t\searrow 0} \frac{M(t)}{t^p} = \infty
$$
any unit vector will do.

Assume that
\begin{equation}
  \label{eq:less-infty}
 \limsup_{t\searrow 0} \frac{M(t)}{t^p} < \infty.
\end{equation}
Since the function $M(uv)/u^pM(v)$ is continuous on the square $(u,v)\in(0,1]\times(0,1]$, we get from \eqref{eq:less-infty} that for any $a\in (0,1]$
$$
  \sup\left\{\frac{M(uv)}{u^pM(v)}:\ u\in (0,1],\ v\in(a,1]\right\} < \infty.
$$
From \eqref{eq:mnz} it follows that
$$
  \sup\left\{\frac{M(uv)}{u^pM(v)}:\ u\in (0,1],\ v\in(0,a]\right\} = \infty.
$$
Thus there are $u_k,v_k\in (0,1]$, $k\in \mathbb{N}$,  such that

\begin{equation}
  \label{eq:seq-M}
  M(v_k) < \frac{1}{k^3},
\end{equation}
and
\begin{equation}
 \label{eq:sec-2}
  \frac{M(u_kv_k)}{u_k^pM(v_k)} > k^3.
\end{equation}

Note that from \eqref{eq:sec-2} and the convexity of $M$ it follows that
$$
  k^3 < \frac{M(u_kv_k)}{u_k^pM(v_k)} \le \frac{u_kM(v_k)}{u_k^pM(v_k)} =\frac{1}{u_k^{p-1}},
$$
so
\begin{equation}
 \label{eq:sec-3}
 u_k < \frac{1}{k^{3/(p-1)}}.
\end{equation}

Let $A_k$, $k\in\mathbb{N}$, be disjoint finite subsets of $\Gamma$ such that $A_k$ has $[ 1/k^2M(v_k)]$ elements, or, in other words,
\begin{equation}
 \label{eq:An-card}
 \sum _{\gamma\in A_k} 1 = \left[ \frac{1}{k^2M(v_k)}\right].
\end{equation}
Note that from \eqref{eq:seq-M}  all $A_k$'s are non-empty.

Define the vector $z$ with support in $\cup A_k$ by
$$
  z(\gamma) = \frac{v_k}{n},\quad \gamma \in A_k,\ k\in[2^{n-1}+1,2^n].
$$
To check that $z\in h_M$, fix $i\in \mathbb{N}$.
\begin{eqnarray*}
  \sigmam (iz) &=& \sum _{k=2}^\infty \sum _{\gamma\in A_k} M(iz(\gamma))\\
       &=& \sum _{n=1}^\infty \sum _{k=2^{n-1}+1}^{2^n} \sum _{\gamma\in A_k} M\left(\frac{iv_k}{n}\right)\quad \mbox{from \eqref{eq:An-card} }\Rightarrow\\
       &=& \sum _{n=1}^\infty \sum _{k=2^{n-1}+1}^{2^n} \left[ \frac{1}{k^2M(v_k)}\right] M\left(\frac{iv_k}{n}\right)\\
       &\le& \sum _{n=1}^\infty \sum _{k=2^{n-1}+1}^{2^n} \frac{M((i/n)v_k)}{k^2M(v_k)}.
\end{eqnarray*}
In order to see that $\sigmam (iz)<\infty$ we estimate the tail $\sum _{n=i}^\infty$ of the above sum. For $n\ge i$ we have $M((i/n)v_k)\le M(v_k)$, so
$$
\sum _{n=i}^\infty \sum _{k=2^{n-1}+1}^{2^n} \frac{M((i/n)v_k)}{k^2M(v_k)}
\le \sum _{n=i}^\infty \sum _{k=2^{n-1}+1}^{2^n} \frac{1}{k^2} = \sum _{k=2^{i-1}+1}^\infty \frac{1}{k^2} < \infty.
$$
Therefore,
$$
  \sigmam (iz) < \infty,\ \forall i\in\mathbb{N} \Rightarrow z\in h_M.
$$

Now,
\begin{eqnarray*}
  \sigmam (nu_{2^n}z) &\ge& \sum _{\gamma\in A_{2^n}} M(nu_{2^n}z(\gamma))\\
     &=& \sum _{\gamma\in A_{2^n}} M(u_{2^n}v_{2^n}) \quad \mbox{from \eqref{eq:An-card} }\Rightarrow\\
     &=&  \left[ \frac{1}{2^{2n}M(v_{2^n})}\right]M(u_{2^n}v_{2^n})\\
     &=& \left( \frac{1}{2^{2n}M(v_{2^n})}\right)M(u_{2^n}v_{2^n})\omega_n,\\
\end{eqnarray*}
where because of \eqref{eq:seq-M}
$$
  \omega_n = \left[ \frac{1}{2^{2n}M(v_{2^n})}\right]\left( \frac{1}{2^{2n}M(v_{2^n})}\right)^{-1} \to 1,\mbox{ as }n\to\infty.
$$
From \eqref{eq:sec-2} it follows that
$$
\frac{M(u_{2^n}v_{2^n})}{M(v_{2^n})} \ge 2^{3n}u^p_{2^n}
$$
and so
$$
\sigmam (nu_{2^n}z) \ge 2^nu^p_{2^n}\omega_n.
$$
Therefore,
$$
  \lim_{n\to\infty} \frac{\sigmam (nu_{2^n}z)}{(nu_{2^n})^p} \ge \lim\frac{2^n}{n^p}=\infty.
$$
Since from \eqref{eq:sec-3} it follows that
$$
nu_{2^n} < \frac{n}{2^{3n/(p-1)}} \to 0,
$$
 \eqref{eq:simple-cond} holds.
\end{proof}

Now we are ready to conclude.

\begin{proof}[\textbf{\emph{Proof of Theorem~\ref{thm:sharp}}}]
    If $p > \alpha_M$ then by \eqref{eq:alfa-def}
    $$
    \sup\left\{\frac{M(uv)}{u^pM(v)}:\ u,v\in (0,1]\right\} = \infty,
    $$
    and Lemma~\ref{lem:one-vector} and Lemma~\ref{lem:subspace} show that $h_M(\Gamma)$ has a  subspace $h_N(\Gamma)$ generated by one vector  with
    $$
      \limsup_{t\to0}\frac{N(t)}{t^p} = \infty.
    $$
    By Proposition~\ref{prop:main} there is no continuous $p$-times G\^ateaux differentiable bump function on $h_N(\Gamma)$. But if $b$ is a continuous $p$-times G\^ateaux differentiable bump on $h_M(\Gamma)$, such that $b(x_0)\neq0$ then the reduction of $x\to b(x-x_0)$ to $h_N(\Gamma)$ is a continuous $p$-times G\^ateaux differentiable bump. This contradiction shows that necessarily
    $$
      \alpha_M \le p.
    $$
    If $\alpha_M < p$ then we have (a) and if $\alpha_M = p$ the two cases of Proposition~\ref{prop:main} cover (b) and (c).
\end{proof}


\begin{thebibliography}{99}

  \bibitem{BF}
      R. Bonic and J. Frampton, \emph{Smooth function on Banach manifolds}, J.
      Math. Mechanics \textbf{15} (1966), 877--898.

  \bibitem{FST} V. Fonf, R. Smith and S. Troyanski, \emph{A note on fragmentability
  and weak-$G_\delta$ sets}, Quarterly J. Math. \textbf{63} (2012), 367--373.


  \bibitem{GMZ}
  A. J. Guirao, V. Montesinos and V. Zizler, \emph{Renormings in
  Banach spaces: A Toolbox}, Monografie Matematyczne \textbf{75}, Birkh\"auser, 2022.

  \bibitem{HJ}
  P. H\'ajek and M. Johanis, \emph{Smooth analysis in Banach spaces}, De Gruyter
  Series in Nonlinear Analysis and Appl. \textbf{19}, De Gruyter, Berlin/Boston, 2014.

  \bibitem{HRS}
  F. Hernandez and B. Rodrigues-Salinas, \emph{Lattice-embedding $L_p$ into Orlicz spaces}, Israel Journal of Mathematics \textbf{90} (1995), 167--188.

  \bibitem{HT-1}
  F. Hernandez and S. Troyanski, \emph{On the representation of the uncountable symmetric basic sets and its applications}, Studia Math. \textbf{107} (1993), 287--304.

  \bibitem{HT}
  F. Hernandez and S. Troyanski, \emph{On G\^ateaux differentiable bump functions}, Studia Math. \textbf{118(2)} (1996), 135--143.

  \bibitem{leung} D. H. Leung, \emph{Some isomorphically polyhedral Orlicz
  sequence spaces}, Israel J. Math. \textbf{87} (1994), 117--128.

  \bibitem{LT1}
      J. Lindenstrauss and L. Tzafriri, \emph{Classical Banach Spaces. I. Sequence
      Spaces}, Springer, Berlin, 1977.

  \bibitem{MNZ}
      R. Maleev, G. Nedev and B. Zlatanov,\emph{ G\^ateaux differentiability of bump
      functions in Banach spaces}, Journal Math. Analysis Appl. \textbf{240} (1999), 311--323.

  \bibitem{MT}
      R. Maleev and S.Troyanski, \emph{Smooth norms in
      Orlicz spaces}, Canadian Math. Bull. \textbf{34} (1991), 74--82.

  \bibitem{RS} B. Rodrigues-Salinas, \emph{Subspaces with symmetric basis in the Orlicz spaces $l_M(I)$}, Ann. Soc. Math. Pol., Ser.I, Commentate Math. \textbf{36} (1996), 201--222.

  \bibitem{Stankov}
  S. Stankov, \emph{Smoothness of bump functions on  $l_p( \Gamma )$ spaces}, MSci thesis, Sofia University, 2016.

  \bibitem{stegal1}
  Ch. Stegall, \emph{Optimization of functions on certain subsets of Banach spaces}, Math. Annal. \textbf{236} (1978), 171--176.

  \bibitem{stegal2}
  Ch. Stegall, \emph{Optimization and differentiation in Banach spaces}, J. Linear Algebr. Appl. \textbf{84} (1986), 191--211.

  \bibitem{Topalova-Zlateva}
  H. Topalova and N. Zlateva, \emph{Perturbation method in Orlicz sequence spaces},
  arXiv:2304.11647 [math.FA], 2023.


  \end{thebibliography}
\end{document}